\documentclass[reqno,12pt]{amsart}

\usepackage[margin=2.5cm]{geometry}
\usepackage[T1]{fontenc}
\usepackage[utf8]{inputenc}
\usepackage[english]{babel}
\usepackage{amsmath,amssymb,amsthm,amsrefs}
\usepackage{ulem}
\usepackage[adobe-utopia]{mathdesign}

\usepackage{orcidlink}

\usepackage{enumitem}
\usepackage{mymacros}
\usepackage{hyperref}
\usepackage{xcolor}
\usepackage{orcidlink}
\usepackage{comment}
\usepackage{tikz}

\usepackage{centernot}
\usepackage{mathtools}

\title{Discrete Random Clark measures and associated inner functions}

\author[C. Bellavita]{Carlo Bellavita \orcidlink{0000-0003-3136-0271}}
\address[Carlo Bellavita]{Dipartimento di Matematica ``F. Enriques''\\
  Dipartimento di Eccellenza MUR 2023-2027\\
Universit\`a degli Studi di Milano\\
Via C. Saldini 50\\
I-20133 Milano}
\email{carlo.bellavita@unimi.it}

\author[N. Chalmoukis]{Nikolaos Chalmoukis \orcidlink{0000-0001-5210-8206}}
\address[Nikolaos Chalmoukis]{Dipartimento di Matematica e Applicazioni, Università
di Milano-Bicocca, via R. Cozzi 53, I-20125 Milano, Italy}
\email{nikolaos.chalmoukis@unimib.it}

\author[G. Lamberti]{Giuseppe Lamberti \orcidlink{0009-0009-0503-0421}}
\address[Giuseppe Lamberti]{Univ. Bordeaux, CNRS, Bordeaux INP, IMB, UMR 5251, F-33400 Talence, France} 
\email{giuseppe.lamberti@math.u-bordeaux.fr}

\thanks{The authors are members of Gruppo Nazionale per l'Analisi
Matematica, la Probabilità e le loro Applicazioni (GNAMPA) of Istituto Nazionale di Alta Matematica (INdAM). The second and third authors are partially supported by the grant ``INdAM-GNAMPA Project'', CUP E53C25002010001, ``Transferring Harmonic Analysis between Discrete Structures and Manifolds''. The second author is partially supported by the research grant “Yields of the ubiquity and the geometry of inner functions (YoungInFun)”, PID2024-160326NA-I00. The third author is partially supported by the project CHAT, ``Complex and Harmonic Analysis in Control Theory'' of the French National Research Agency, ANR-24-CE40-5470.}


\begin{document}

\begin{abstract}
We study a class of random inner functions $\varphi$ whose Clark measure at $1$ is the weighted sum of point masses supported on independent uniformly distributed points of $\mathbb T$. Our first result shows that $\varphi$ is almost surely a Blaschke product. We then investigate when \(\varphi\) admits angular derivative almost surely and we provide a \( 0 - 1\) law. These conditions have a direct interpretation in terms of the other Clark measures associated with $\varphi$.
Finally, we obtain quantitative estimates for the zeros of $\varphi$, proving that, in suitable regimes, their distribution satisfies summability conditions stronger than the classical Blaschke condition.
    
\end{abstract}

\maketitle

\section{Introduction and main results}

Let $\varphi: \mathbb{D} \to \mathbb{D}$ be an analytic self-map of the unit disk in the complex plane. For each $\alpha \in \mathbb{T} = \partial\mathbb{D}$, the \textit{Clark measure} $\sigma_\alpha = \sigma_{\varphi,\alpha}$ is defined via the Herglotz representation
\begin{equation*}
\text{Re} \left( \frac{\alpha + \varphi(z)}{\alpha - \varphi(z)} \right) = \int_{\mathbb{T}} \frac{1 - |z|^2}{|\zeta - z|^2} \, d \sigma_{\varphi,\alpha}(\zeta), \quad z \in \mathbb{D}.
\end{equation*}

Originally introduced by Clark \cite{Clark72} and further developed by Aleksandrov (see \cites{Ale87, Ale89}) these measures have been extensively studied, especially regarding their applications to perturbative operator theory \cites{Clark72, Simon1986,Simon1994}. For a detailed review of the basic properties of this family of measures, we refer the reader to \cites{BN2024,PS06,Saksman,Rosscauchy}.

In \cite{Huangsaksman}, the authors studied a random analytic self-map $\varphi$ defined via its Clark measure $\sigma_1$. In particular they considered $\sigma_1$ to be a Gaussian multiplicative chaos (GMC) measure and they proved that $\varphi$ is almost surely an infinite Blaschke product.

In this paper, we investigate a new random self-map of the unit disk defined via discrete random Clark measures. We now describe the construction in more detail. 
Let $(\vartheta_n)_{n\in\mathbb{N}}$ be a sequence of independent random variables, each uniformly distributed on $[0,2\pi)$, and set $\zeta_n(\omega)=e^{i\vartheta_n(\omega)}\in\mathbb{T}$. Given a sequence of non-negative weights $(c_n)_{n\in\mathbb{N}}$ satisfying the normalization condition $\sum_{n \in \bN} c_n=1$, we define the random measure
\begin{equation}\label{mu:def}
\mu=\sum_{n\in\mathbb{N}} c_n\,\delta_{\zeta_n}.
\end{equation}
This is a random measure in the following standard sense: $\mu$ is a measurable map from a probability space $\Omega$ to the space of finite non-negative measures on $\bT$, equipped with the $\sigma$-algebra generated by the evaluation maps $\nu\mapsto \nu(B)$ for Borel sets $B\subset\mathbb{T}$. Random measures of this type have been studied in various contexts (see, for instance, \cite{Kahane85}). Informally, $\mu$ is obtained by placing masses at points distributed uniformly on the unit circle. For background on random measures, we refer to \cite{Kallenberg83}. Finally, $\mu$ is singular with respect to Lebesgue measure on $\mathbb{T}$, whose normalized version we denote by $dm$.

We consider the random function $\varphi : \bD \to \bD$ defined such that its Clark measure at $\alpha = 1$ is $\mu$, i.e., $\mu=\sigma_{\varphi,1}$. Specifically, through standard computations \cite[Section 9]{Rosscauchy}, $\varphi$ is given by
\begin{equation}\label{phi:def}
    \varphi(z) = \frac{H\mu(z)-1}{H\mu(z)+1}, \quad \text{where} \quad H\mu(z) = \int_{\bT} \frac{\zeta+z}{\zeta-z} \, d\mu(\zeta). 
\end{equation}
It follows immediately that the random analytic self-map $\varphi$ is an inner function, which means that 
\[
|\varphi(e^{i\vartheta})| =\lim_{r \to 1} |\varphi(re^{i\vartheta})|=1
\]
for almost every $\vartheta \in [0,2\pi)$. Indeed, it is a classical fact \cite[Theorem 2.2]{Saksman} on Clark measures that $\varphi$ is an inner function if and only if the Clark measure $\sigma_{\varphi,1}$ is singular with respect to the Lebesgue measure.

Before presenting our results, we stress that our model is fundamentally different from the one introduced in \cite{Huangsaksman}. Although both GMC measures and countable weighted sums of Dirac masses are singular with respect to Lebesgue measure, the former are almost surely non-atomic (see \cite[Theorem 2.6]{RV14}), whereas the measure $\mu$ considered here is purely atomic, hence discrete.

\bigskip
We recall that any inner function $\varphi$ admits the factorization \cite[Theorem 5.5]{garnett06}
\[
\varphi(z) = B_Z(z)S_\eta(z),
\]
where $B_Z$ is a Blaschke product
\[
B_Z(z) = e^{ic} z^m \prod_{|z_n| \neq 0} \frac{|z_n|}{z_n} \frac{z_n - z}{1 - \overline{z_n}z}
\]
with $c \in \mathbb{R}$, $m \geq 0$, $Z=(z_n)_{n \in \bN}$ the sequence of zeros of $B_Z$ such that $\sum_{n \in \bN} (1 - |z_n|) < \infty$ and $S_\eta$ is a singular inner function
\[
S_\eta(z) = \exp \left( -\int_{\mathbb{T}} \frac{e^{i\vartheta} + z}{e^{i\vartheta} - z} \, d\eta(\vartheta) \right),
\]
where the measure $d\eta$ is positive and singular with respect to the Lebesgue measure.

Our first main result says that the random inner function $\varphi$ defined in \eqref{phi:def} does not have any singular factor.
\begin{thm} \label{THM1}
    Let $\mu$ be a random measure defined as in \eqref{mu:def}. Then the random inner function $\varphi$ defined as in \eqref{phi:def} is almost surely a Blaschke product.
\end{thm}
In the deterministic setting, Blaschke products are dense in the class of all inner functions (see \cite[Theorem 6.4 and Corollary 6.5]{garnett06}). Moreover, our result combined with \cite[Theorem 1]{Huangsaksman} suggests that, in the probabilistic setting where one focuses on almost sure behavior, singular inner functions are atypical and appear only on events of probability zero.
Furthermore, we highlight that in general it is quite difficult to understand many properties of an inner function $\varphi$ knowing only one of its Clark measures. We point out that there exist singular inner functions whose associated Clark measures are infinite sums of Dirac masses, one such example is $\varphi(z) = \exp\left(\frac{z+1}{z-1}\right)$, see \cite{Saksman}.

\bigskip
Let $\varphi$ be an analytic self-map of the unit disk $\mathbb{D}$. With the symbol $\angle \lim_{z \to \zeta}$, we denote the non-tangential limit, that is, $\lim_{z \to \zeta}$ where $z$ is taken in the \textit{Stoltz region} $\Gamma_{\gamma}(\zeta)=\{w \in \bD: |w-\zeta|<\gamma(1-|w|)\}$, with $\gamma >1$. We say that $\varphi$ has \textit{angular derivative} (in the sense of Carathéodory) at a point $\zeta \in \mathbb{T}$ if the following two conditions hold:

\begin{enumerate}[label=(\roman*)]
    \item The non-tangential limit $\kappa = \angle \lim_{z \to \zeta} \varphi(z)$ exists and belongs to the boundary $\mathbb{T}$.
    \item The non-tangential limit $\angle \lim_{z \to \zeta} \varphi'(z)$ exists (and it is necessarily finite).
\end{enumerate}

While the existence of the angular derivative is guaranteed if $\varphi$ extends analytically across $\zeta$ with $|\varphi(\zeta)| = 1$, the general case is more subtle. Ahern and Clark established a criterion for the existence of angular derivatives for inner functions (see for instance \cites{AC1970,AC1974}). These results were later extended to general bounded analytic functions \cite{FM2008}.

\medskip

Regarding the existence of the angular derivative of the random inner function $\varphi$ defined through \eqref{phi:def} we established the following result.

\begin{thm} \label{THM2}
Let $\varphi$ be the random inner function defined as in \eqref{phi:def}.
\begin{enumerate}[label=(\roman*)]
    \item Suppose that $\sum c_n^{1/2} < \infty$. Then, for every fixed $\xi \in \bT$, the random inner function $\varphi$ admits almost surely angular derivative at $\xi$.
    \item Suppose that $\sum c_n^{1/2} =\infty$. Then, for every fixed $\xi \in \bT$, the random inner function $\varphi$ admits almost never angular derivative at $\xi$.
\end{enumerate}
\end{thm}

\bigskip

An important consequence of Theorem \ref{THM2} is that it yields information on the remaining Clark measures associated with the random inner function \(\varphi\). Recall that, by the Lebesgue decomposition theorem, each Clark measure \(\sigma_\alpha\) admits a decomposition of the form
\[
\sigma_\alpha = \eta_\alpha + d_\alpha,
\]
where \(\eta_\alpha\) is singular continuous and \(d_\alpha\) is purely atomic (discrete). We can then state the following corollary.

\begin{cor}\label{THM3}
    Let $\varphi$ be the random inner function defined as in \eqref{phi:def}, and for $\alpha \in \bT$, let $\sigma_\alpha$ be its Clark measure in $\alpha$.
    \begin{enumerate}[label=(\roman*)]
        \item Suppose that $\sum c_n^{1/2} < \infty$. Then, almost surely, we have that $\sigma_\alpha=d_\alpha$ for $m- $a. e. $\alpha \in \mathbb{T}. $ 
        \item Suppose that $\sum c_n^{1/2}= \infty$. Then, almost surely,  we have that $\sigma_\alpha=\eta_\alpha$ for $m- $ a.e. $\alpha \in \mathbb{T} \setminus \{1 \} $.
    \end{enumerate}
\end{cor}

It is useful to recall that Clark measures arise naturally as spectral measures of the associated \textit{Clark operators}. Let \(\varphi\) be an inner function with \(\varphi(0)=0\), and consider the model space \(K_\varphi = H^2 \ominus \varphi H^2\), where \(H^2\) is the classical Hardy space. Denote by
\[
S_\varphi = P_{K_\varphi}\, S\!\restriction_{K_\varphi}
\]
the compression to \(K_\varphi\) of the unilateral shift \(S\), given by \((Sf)(z)=zf(z)\). For each \(\alpha \in \mathbb{T}\), define the rank-one unitary perturbation
\begin{equation}
U_\alpha = S_\varphi + \alpha \,\langle \cdot, S^*\varphi  \rangle \,\mathbf{1}
\end{equation}
acting on \(K_\varphi\) , where \(S^*\) is the backward shift operator and \(\mathbf{1}\) is the unit constant function. The operator \(U_\alpha\) is unitary and cyclic \cite[Theorem 11.4]{Rossmodel}. Moreover, its point spectrum consists of those \(\zeta \in \mathbb{T}\) for which \(\varphi(\zeta)=\alpha\) and the angular derivative satisfies \(|\varphi'(\zeta)|<\infty\). A standard result (see, e.g., \cite{Clark72}) asserts that the spectral measure of \(U_\alpha\) coincides with the Clark measure \(\sigma_\alpha\). In this way, Corollary \(\ref{THM3}\) can be viewed as describing spectral properties of random Clark operators associated to $\varphi$.

\bigskip
It is natural to investigate the behavior of the sequence of zeros $\Lambda=(\lambda_k)_{k \in \bN}$ of the random Blaschke product $\varphi$. In \cite{Huangsaksman}, the authors established a sharp condition on the exponent $\beta$ ensuring the convergence of the sum $\sum_k (1-|\lambda_k|)^\beta$. It is worth noting that, in that setting the above sum diverges for $\beta=3/4$, whereas in our framework convergence may occur. More precisely, we prove the following theorem.

\begin{thm}\label{THM4}
    Let $\varphi$ be the random inner function defined as in \eqref{phi:def}, $\Lambda=(\lambda_k)_{k \in \bN}$ the sequence of its zeros and $1/2 \leq \alpha <1$.  If $(c_n)_{n \in \bN} \in \ell^\alpha$, then for every $\varepsilon >0$,
    \[
    \sum_{k \in \bN} (1-|\lambda_k|)^{\alpha + \varepsilon} < \infty \quad \text{almost surely.}
    \]
\end{thm}

We prove a partial counterpart of Theorem \ref{THM4}. After reindexing the pairs \((c_n,\zeta_n)\), we may assume without loss of
generality that the weights \((c_n)_{n \in \mathbb{N}}\) are strictly positive and
decreasing. This does not change the associated random measure, and hence it
gives rise to the same random inner function \(\varphi\).

\begin{thm}\label{Prop1}
	Let $\varphi$ be the random inner function defined as in \eqref{phi:def} and $\Lambda=(\lambda_k)_{k \in \bN}$ the sequence of its zeros. Suppose that \(c_n\) is a strictly positive, decreasing sequence such that
	\begin{equation} \label{log-regularity} \tag{LR}
	\lim_{n \to \infty} \frac{-\log c_n}{\log n} = p,
	\end{equation}
	with $p\geq 1$. Then for every $\varepsilon >0$
	\[
	\sum_{k \in \bN} (1-|\lambda_k|)^{2-p-\varepsilon} = \infty \quad \text{almost surely.}
	\]
\end{thm}

We emphasize that the above Theorem provides information only in the case $1 \leq p<2$. 
For $p=1$, the result translates to
\[
\sum_{k \in \bN} (1-|\lambda_k|)^{1-\varepsilon} = \infty \quad \text{almost surely},
\]
which means that the zeros of the random Blaschke product $\varphi$ satisfy (of course) the Blaschke condition, but nothing stronger than that. An example of such sequences is $c_n=(n\log^2(n))^{-1}$.

Furthermore, we mention that if one defines the exponent of convergence of $(c_n)_{n \in \bN}$ by
	\[
	\alpha_c=\inf\left\{\alpha>0:\sum_{n \in \bN} c_n^\alpha<\infty\right\},
	\]
	then under the assumption that the limit 
	\[
	\lim_{n\to\infty}\frac{-\log c_n}{\log n}=p
	\]
	exists, one has
	\[
	\alpha_c=\frac 1 p.
	\]
This means that, under the same hypothesis, Theorem \ref{Prop1} can be reformulated in the following way:
\[
\sum_{k \in \bN} (1-|\lambda_k|)^{2-\frac{1}{\alpha_c}-\varepsilon} = \infty \quad \text{almost surely.}
\]

\begin{rem}
Let $(\vartheta_n)_{n\in\mathbb{N}}$ be a sequence of independent random variables, each uniformly distributed on $[0,2\pi)$, and set $\zeta_n(\omega)=e^{i\vartheta_n(\omega)}\in\mathbb{T}$. Given a sequence of non-negative weights $(c_n)_{n\in\mathbb{N}}$ satisfying the normalization condition $\sum_{n \in \bN} c_n=1$, we define the random measure
\begin{equation}\label{eq: with Bernoulli}
    \nu = \sum_{n \in \mathbb{N}} c_n \varepsilon_n \delta_{\zeta_n}
\end{equation}
where $(\varepsilon_n)_{n \in \mathbb{N}}$ is a sequence of independent (also with respect to the sequence of random variables $(\vartheta_n)_{n \in \bN}$) Bernoulli random variables. In this case, the total variation $\|\nu\|$ is not necessarily equal to $1$. Consequently, the associated random inner function $\psi$ has a different expression. If we denote by $H\nu$ the Herglotz transform of $\nu$, we have
\begin{equation}\label{newvarpsi}
    \int_{\mathbb{T}} \frac{\zeta+z}{\zeta-z} \, d\nu(\zeta) = \frac{1+\psi(z)}{1-\psi(z)} + iA
\end{equation}
where $A = 2\text{Im}(\psi(0))/ |1-\psi(0)|^2$ is a real constant. Consequently we obtain
\begin{equation}\label{newvarphi}
    \psi(z) = \frac{H\nu(z) - iA - 1}{H\nu(z) - iA + 1}.
\end{equation}

With small modifications, Theorem \ref{THM1}, Theorem \ref{THM2} and Theorem \ref{THM4} can be extended to the class of random discrete Clark measures defined as in \eqref{eq: with Bernoulli} and the associated random inner function $\psi$. It is likely that the same happens also for Theorem \ref{Prop1}, but we will not pursue this here.
\end{rem}
Furthermore, we point out that the assumption $\sum_n c_n = 1$ has been made solely to simplify the computations: all the results presented here in fact hold under the weaker assumption $\sum_n c_n < \infty$.

\bigskip
The paper is structured as follows. Section \ref{sect:thm1} is mainly devoted to the proof of Theorem \ref{THM1} while in Section \ref{sect:thm2} we prove Theorem \ref{THM2} and we further develop the connections to Clark operators. In Section \ref{sect:thm4} we prove Theorem \ref{THM4} and Theorem \ref{Prop1}.

\subsection{Notation} If $f$ and $g$ are positive expressions, we will write $f \lesssim g$ if there exists $C>0$ such that $f\leq C g$, where $C$ does not depend on the parameters behind $f$ and $g$. We will simply write $f \simeq g$ if $f \lesssim g$ and $g \lesssim f$. Finally when $f$ and $g$ are expressions for which we can consider the limit of their quotient in a point, with $f \sim g$ we mean that $\lim f/g =1$ in that point.

\section{Proof of Theorem \ref{THM1}}\label{sect:thm1}

Recall that the random function $\varphi$ is inner because its associated Clark measure is singular. Before addressing Theorem \ref{THM1}, we establish the following property regarding its boundary behavior.

\begin{prop}\label{Prop3.1}
Let $\varphi$ be the random inner function defined as in \eqref{phi:def}.
Then, for every fixed $\zeta\in\mathbb T$, the radial limit of the random inner function $\varphi$ at $\zeta$ is almost surely unimodular.
\end{prop}

\begin{proof}
For $\omega \in \Omega$, we denote by $\varphi_\omega$ the realization of the random variable $\varphi$ corresponding to $\omega$.
For a fixed $\xi\in\mathbb T$, we define the exceptional event
\[
E_\xi
=
\left\{
\omega \in \Omega:
\lim_{r\to1^-}|\varphi_\omega(r\xi)|\neq 1
\right\}.
\]
Equivalently, since $|\varphi_\omega|\leq 1$ in $\mathbb D$,
\[
E_\xi
=
\left\{
\omega \in \Omega:
\liminf_{r\to1^-}|\varphi_\omega(r\xi)|<1
\right\}.
\]
For every realization $\omega$ in the probability space $\Omega$, the function $\varphi_\omega$ is inner.
Hence, by Fatou's theorem
\[
\lim_{r\to1^-}|\varphi_\omega(r\xi)|=1
\]
for $m-$a.e. $\xi\in\mathbb T$.
Therefore,
\[
m\left(\{\xi\in\mathbb T:\omega\in E_\xi\}\right)=0
\]
for almost every $\omega \in \Omega$. By Fubini's theorem,
\[
0
=
\mathbb E\left[
m\left(\{\xi\in\mathbb T:\omega\in E_\xi\}\right)
\right]
=
\int_{\mathbb T}\mathbb P(E_\xi)\,dm(\xi).
\]

It remains to use rotational invariance. Since the random variables
$(\zeta_n)_{n\geq1}$ are i.i.d. and uniformly distributed on $\mathbb T$, the
law of $(\zeta_n)_{n\geq1}$ is invariant under rotations. Consequently, for
any $\xi,\zeta\in\mathbb T$,
\[
\mathbb P(E_\xi)=\mathbb P(E_\zeta).
\]
Thus the function $\xi\mapsto \mathbb P(E_\xi)$ is constant and must be zero. Therefore,
\(
\mathbb P(E_\zeta)=0
\)
for every fixed $\zeta\in\mathbb T$. Hence, for every fixed $\zeta\in\mathbb T$,
\[
\lim_{r\to1^-}|\varphi(r\zeta)|=1
\]
almost surely.
\end{proof}

\begin{rem}
	The proof of Proposition \ref{Prop3.1} does not rely on the specific discrete random model described in \eqref{mu:def}, but only on the fact that the random variables $\zeta_n$ are rotationally invariant. Consequently, this argument holds for a broader class of random Clark measures.

\end{rem}

\bigskip
To prove Theorem \ref{THM1}, we will use the following Lemma, which can be found in \cite{Huangsaksman}.
\begin{lem}[\cite{Huangsaksman}*{Lemma 3.1}]\label{Lemma 1}
Let $\varphi$ be a random inner function on the unit disc. If for some $a>1$ 
\begin{equation}\label{eq 1}
    \sup_{z \in \bD} \mathbb{E}\left[\left(\log \frac{1}{|\varphi(z)|}\right)^{a}\right]<\infty,
\end{equation}
 then $\varphi$ is almost surely a Blaschke product.
\end{lem}
We want to apply the above lemma to the random inner function $\varphi$ defined as in \eqref{phi:def}.
Since $\Re(H\mu(z)) \geq 0$ and $\varphi(0)=0$, the result follows if we can show that
\[
\sup_{|z| \geq \delta} \bE \left[ \left( \log^+ \frac{1}{|H\mu(z)-1|} \right)^{a} \right] < \infty
\]
for some $a > 1$ and $\delta > 0$. Furthermore, the rotational invariance of $\mu$ allows us to restrict our attention to the radial segment $[ \delta, 1)$: indeed
\[
\bE\left[ H\mu(z) \right]=\bE\left[ H\mu(e^{is}z) \right]
\]
for every $0<s<2\pi$.
Thus, it is sufficient to demonstrate that
\begin{equation*}
    \sup_{\delta \leq r < 1} \bE \left[ \left( \log^+ \frac{1}{|H\mu(r)-1|} \right)^{a} \right] < \infty.
\end{equation*}
Using the lower bound $|H\mu(r)-1| \geq |\Im(H\mu(r))|$, the above condition is implied by
\[
\sup_{\delta \leq r < 1} \bE \left[ \left( \log^+ \frac{1}{|\Im(H\mu(r))|} \right)^{a} \right] < \infty,
\]
which, in our setting, is equivalent to showing
\begin{equation} \label{eq:Saksman variation}
    \sup_{\delta \leq r < 1} \bE \left[ \left( \log^+ \frac{1}{\left\vert \sum_{n \in \mathbb{N}} c_n \frac{2r \sin(\vartheta_n)}{1 - 2r \cos(\vartheta_n) + r^2} \right\vert} \right)^{a} \right] < \infty.
\end{equation}
The rest of this section is dedicated to the proof of \eqref{eq:Saksman variation}. In order to achieve that, we recall the classical Van der Corput Lemma about oscillatory integrals of
the form
\[
I(\lambda)=\int_{a}^b e^{i\lambda u(x)}\, dx, \qquad \lambda \in \bR.
\]

\begin{thm}\label{Van der Corput}
Let $-\infty<a<b<\infty$. 
\begin{itemize}
    \item[(i)]  Let $k\geq 2$. For every real valued function $u \in C^k(\mathbb{R})$ that satisfies $|u^{(k)}(t)| \geq 1$ for all $t\in (a,b)$, we have
\begin{equation*}\label{eq 4}
\left\vert   \int_{a}^b e^{i\lambda u(x)}\,dx\right \vert\leq 12 |\lambda|^{-\frac{1}{k}}\,  \quad \text{for all } \lambda \neq 0.
\end{equation*}
\item[(ii)] Let $k = 1$. For every function $u \in C^1(\mathbb{R})$ such that $|u'(t)| \geq 1$ and $u$ is monotonic in $(a,b)$, we have
\begin{equation*}\label{eq 5}
   \left\vert   \int_{a}^b e^{i\lambda u(x)}\, dx\right \vert\leq 3|\lambda|^{-1}\,  \quad \text{for all } \lambda \neq 0.
\end{equation*}
\end{itemize}
\end{thm}
For a proof, we refer to \cite{Stein93}*{Proposition 2, Section VIII.1}.

\bigskip
The following Lemma  is a direct application of the Van der Corput lemma.

\begin{lem}\label{Lem app. VDC}
There exists a positive absolute constant $C$ such that for any $r \in [\delta, 1)$ and $|\lambda|\geq1$,
\begin{equation}\label{eq 6}
 \left\vert \int_{-\pi}^{\pi} e^{i\lambda \frac{\sin(\vartheta)}{1-2r\cos(\vartheta)+r^2}} d\vartheta \right\vert \leq C|\lambda|^{-1/2}.
\end{equation} 
\end{lem}

\begin{proof}
First of all by symmetry it suffices to estimate the integral from $0$ to $\pi$. We want to use Theorem \ref{Van der Corput} with 
\[
u(\vartheta)=\frac{\sin(\vartheta)}{1-2r\cos(\vartheta)+r^2}.
\]
For \(r\) sufficiently close to \(1\), set
\[
a_1=\frac{1-r}{2},
\qquad
a_2=\frac{4}{3}(1-r),
\qquad
a_3=
\arccos\left(
\frac{6r^2-r^4-1}{2r(1+r^2)}
\right).
\]
Then \( a_0=0<a_1<a_2<a_3<\pi = a_4\). Moreover, \(u_r'\) is monotone on each of the intervals
\[
[0,a_1],\qquad [a_2,a_3],\qquad [a_3,\pi],
\]
and there exists a constant \(\kappa>0\), independent of \(r\), such that
\(
|u_r'(t)|>\kappa
\)
on these intervals. In addition,
\[
u_r''(t)\le -\kappa,
\qquad t\in[a_1(r),a_2(r)],
\]
for all \(r\) sufficiently close to \(1\).
Therefore, due to Van der Corput Lemma and by splitting the integral in the following four intervals, we have  that
\begin{align*}
 \left \vert \int_{-\pi}^{\pi} e^{i\lambda \frac{\sin(\vartheta)}{1-2r\cos(\vartheta)+r^2}}d\vartheta   \right\vert=&2 \left \vert \int_{0}^{\pi} e^{i\lambda \frac{\sin(\vartheta)}{1-2r\cos(\vartheta)+r^2}}d\vartheta   \right\vert\\
 \leq& 2 \sum_{n=1}^{4} \Big| \int_{a_{n-1}}^{a_n} e^{i\lambda \frac{\sin(\vartheta)}{1-2r\cos(\vartheta)+r^2}}d\theta \Big| \\
 \leq& \frac{48}{|\kappa \lambda|^{\frac{1}{2}}}+ \frac{18}{|\kappa \lambda|}\leq \frac{C}{| \lambda|^{\frac{1}{2}}}\ . 
\end{align*}
This ends the proof.
\end{proof}

To obtain our result, we will use Ess\'een's concentration inequality (see \cite{Esseen66} or \cite{Tao12}*{Exercise 2.2.11}), which is formulated in terms of the characteristic function of a random variable. For the reader's convenience, we recall the relevant definitions and state the inequality here. Given a real-valued random variable \(X\), its \textit{characteristic function} \(\phi_X\) is defined by
\begin{equation}
    \phi_X(t) = \bE\!\left[e^{itX}\right] = \int_{\bR} e^{itx}\, d\nu_X(x),
\end{equation}
where \(\nu_X\) denotes the distribution of \(X\). We may now state the inequality.

\begin{thm}[Esséen concentration inequality]
    Let $X$ be a random variable taking values in $\bR$. Then for any $r>0$, $\varepsilon >0$, there exists $C_{\varepsilon}$ depending only on $\varepsilon$ such that
    \begin{equation}
        \sup_{x_0 \in \bR} \bP(|X-x_0| \leq r) \leq C_{\varepsilon}r \int_{|t|\leq \frac{\varepsilon}{r}} |\phi_X(t)|\,  dt.
    \end{equation}
\end{thm}

\begin{proof}[Proof of Theorem \ref{THM1}]
    We need to verify that \eqref{eq:Saksman variation} holds.

We first prove that there exists a positive constant $C'$ such that for $\varepsilon < 1$,
\begin{equation}\label{eq 7}
 \mathbb{P}\left( \left\vert \sum_{n \in \mathbb{N}}c_n\frac{2r\sin(\vartheta_n)}{1-2r\cos(\vartheta_n)+r^2}\right\vert\leq \varepsilon \right)  \leq C' \sqrt{\varepsilon}. 
\end{equation}
We can suppose without loss of generality that $c_0 \neq 0$ and $\varepsilon < 4\pi \delta c_0< 4\pi c_0<1$. 
First of all notice that the characteristic function $\phi_S(\lambda)$ of the random variable 
\[
S =\sum_n 2rc_n \frac{\sin(\vartheta_n)}{1-2r\cos(\vartheta_n)+r^2}\, ,
\]
coincides with
\[
\phi_S(\lambda) = \prod_{n}\phi_0(2rc_n\lambda)\, ,
\]
where 
\[
\phi_0(\lambda)=\mathbb{E}\left( e^{i \lambda \frac{\sin(\vartheta)}{1-2r\cos(\vartheta)+r^2}}\right), 
\]
since $\vartheta_n$'s are independent random variables.
Since for every random variable $X$ we have that $|\phi_X(\lambda)| \leq  1$, it follows that
\[
|\phi_S(\lambda)|\leq |\phi_0(2rc_0\lambda)|\, .
\]
Hence, by Ess\'een inequality, we find that
\begin{align*}
  \mathbb{P}\left( \left\vert S\right\vert\leq \varepsilon \right) & \leq  C_{2\pi}\, \varepsilon \int_{-\frac{2\pi}{\varepsilon}}^{\frac{2\pi}{\varepsilon}}|\phi_S(\lambda)|\,d\lambda
  \leq C_{2\pi}\, \varepsilon\int_{-\frac{2\pi}{\varepsilon}}^{\frac{2\pi}{\varepsilon}}|\phi_0(2rc_0\lambda)|\,d\lambda
  = \frac{C_{2\pi}}{2rc_0}\varepsilon \int_{-\frac{4\pi rc_0}{\varepsilon}}^{\frac{4\pi rc_0}{\varepsilon}}|\phi_0(\lambda)|\,d\lambda\\
  & \leq \frac{C_{2\pi}}{rc_0}\varepsilon \Big( \int_0^1 |\phi_0(\lambda)| d\lambda  + \int_{1}^{\frac{4\pi rc_0}{\varepsilon}} \frac{C}{\sqrt{\lambda}} d\lambda\Big) \\
  & \leq C_2\varepsilon + C_3 \sqrt{\varepsilon} \\
  & \leq C'\sqrt{\varepsilon}.
\end{align*}
By Kolmogorov formula, we have that
 \begin{align*}
     \mathbb{E}\left[\left(\log^+\frac{1}{\left\vert \sum_{n} c_n\frac{2r\sin(\vartheta_n)}{1-2r\cos(\vartheta_n)+r^2}\right\vert}\right)^a\right]& = a\int_0^{+\infty} \mathbb{P}\left( \log^+\frac{1}{\left\vert \sum_{n}c_n\frac{2r\sin(\vartheta_n)}{1-2r\cos(\vartheta_n)+r^2}\right\vert}\geq t\right) t^{a-1}dt\\
     & = a\int_0^{+\infty} \mathbb{P}\left( \left\vert \sum_{n}c_n\frac{2r\sin(\vartheta_n)}{1-2r\cos(\vartheta_n)+r^2}\right\vert\leq e^{-t}\right) t^{a-1}dt\\
     & \leq \int_0^{-\log(4\pi rc_0)}t^{a-1}dt + C' \int_{-\log(4\pi rc_0)}^{+\infty} e^{-\frac{t}{2}} t^{a-1} dt\\
     &\leq \int_0^{-\log(4\pi \delta c_0)}t^{a-1}dt + C' \int_{-\log(4\pi \delta c_0)}^{+\infty} e^{-\frac{t}{2}} t^{a-1} dt < \infty . 
     \end{align*}
Since the last expression does not depend on $r$, the theorem is now proved.
\end{proof}


\section{Proof of Theorem \ref{THM2}} \label{sect:thm2}

We mention the relationship between the existence of the angular derivative for an inner function $\varphi$ and the properties of the Clark measures $\sigma_{\varphi, \alpha}$. The following theorem characterizes the presence of atoms in the Clark measure in terms of the angular derivative.

\begin{lem}[\cite{Saksman}, Theorem 3.1] \label{lem:saksman1}
The inner function $\varphi$ possesses an angular derivative at a point $\xi \in \mathbb{T}$ if and only if for some $\alpha \in \bT$ the Clark measure $\sigma_{\varphi, \alpha}$ has an atom at $\xi$. In this case, $\alpha = \angle \lim_{z \to \xi} \varphi(z)$ and the derivative satisfies
\[
\angle \lim_{z \to \xi} |\varphi'(z)| = \frac{1}{\sigma_{\varphi, \alpha}(\{\xi\})}.
\]
\end{lem}

The finiteness of the angular derivative can also be characterized by the integrability of the Clark measure.

\begin{lem}[\cite{Saksman}, Lemma 3.4]
Let $\xi, \alpha \in \mathbb{T}$. Then $|\varphi'(\xi)| < \infty$ if and only if
\begin{equation}
    \sigma_{\varphi, \alpha}(\{\xi\}) > 0 \quad \text{or} \quad \int_{\mathbb{T}} \frac{1}{|\xi - \zeta|^2} \, d\sigma_{\varphi, \alpha}(\zeta) < \infty.
\end{equation}
\end{lem}

Recall that for every fixed $\xi \in \bT$ we have that 
\[
\bP(\exists n \in \bN \text{ such that } \zeta_n=\xi)=0.
\]
Therefore, for every $\xi \in \bT$ we have
\[
\int_{\bT} \frac{d\mu(\zeta)}{|\xi-\zeta|^2} = \sum_{n \in \bN} \frac{c_n}{|\xi-\zeta_n|^2}.
\]
The remainder of this section is devoted to the study of the convergence of the random series above. To this end, we recall the following standard result (see, for instance, \cite{billingsley95}).

\begin{thm}[Kolmogorov's three series theorem] \label{kolmogorov 3-series}
Let $(X_n)_{n \in \bN}$ be a sequence of independent random variables and for any $A>0$ let $Y_n := X_n \chi_{\{|X_n| \leq A\}}$. Consider the series
    \begin{enumerate}[label=(\roman*)]
        \item $\sum_n \bP(|X_n| > A)$,
        \item $\sum_n \bE[Y_n]$,
        \item $\sum_n \Var[Y_n]$.
    \end{enumerate}
    In order that $\sum_n X_n$ converges almost surely, it is necessary that the three series converge for all positive $A$ and sufficient that they converge for some positive $A$.
\end{thm}

We can now prove our second result.
\begin{proof}[Proof of Theorem \ref{THM2}]

The following argument is an adaptation of the proof of \cite{HL25}*{Lemma 4.1}. Since the random sequence $(\zeta_n)_{n \in \bN}$ we are considering is invariant by rotations, we can assume without loss of generality that $\xi=1$. Set
   \[
   X_n = \frac{c_n}{|1-\zeta_n|^2},
   \]
   and note that the $X_n$ are independent. Since $|1-\zeta_n|^2 = 2 - 2\cos(\vartheta_n)$, we have
   \[
   \bP(X_n > 1) = \bP\Big( c_n > 2 - 2\cos(\vartheta_n) \Big) = \bP\Big( \cos(\vartheta_n) > 1 - \tfrac{c_n}{2} \Big).
   \]
   Writing $u_n = 1 - \tfrac{c_n}{2}$ and using that $\vartheta_n$ is uniform on $[0, 2\pi)$, this probability splits over the two arcs where $\cos(\vartheta_n) > u_n$, giving
   \[
   \bP(X_n > 1) = \bP\big(0 < \vartheta_n < \arccos(u_n)\big) + \bP\big(2\pi - \arccos(u_n) < \vartheta_n < 2\pi\big) = \frac{\arccos(u_n)}{\pi}.
   \]
   As $n \to \infty$ we have $u_n \to 1$, so that $\arccos(u_n) \sim (1 - u_n^2)^{1/2}$. Since moreover $1 - u_n^2 \sim c_n$, we conclude
   \[
   \bP(X_n > 1) \sim \frac{c_n^{1/2}}{\pi}.
   \]
   In particular, $\sum_n c_n^{1/2} = \infty$ forces $\sum_{n \in \bN} \bP(X_n > 1) = \infty$, and Theorem \ref{kolmogorov 3-series} then yields $$\bP\big(\sum_n X_n < \infty\big) = 0.$$ This proves the second part of the claim.

\medskip
Suppose now that $\sum_n c_n^{1/2} < \infty$. This condition immediately gives $\sum_{n \in \bN} \bP(X_n > 1) < \infty$, which is precisely condition (i) of Theorem \ref{kolmogorov 3-series}. It remains to verify conditions (ii) and (iii) for the truncated variables
\[
Y_n = \frac{c_n}{|1-\zeta_n|^2}\, \chi_{\{X_n \leq 1\}}.
\]
We first estimate $\bE[Y_n]$. As $\vartheta \mapsto |1-e^{i\vartheta}|^2$ is even, the integral over the relevant arc reduces to
\begin{align}\label{Cond(ii)}
    \bE[Y_n] & = \frac{1}{2\pi}\int_{\arccos(u_n)}^{2\pi-\arccos(u_n)} \frac{c_n}{|1-e^{i\vartheta}|^2}\, d\vartheta 
     = \frac{1}{\pi} \int_{\arccos(u_n)}^{\pi} \frac{c_n}{|1-e^{i\vartheta}|^{2}}\, d\vartheta \nonumber\\
     & = \frac{c_n}{2\pi} \int_{\arccos(u_n)}^\pi \frac{d\vartheta}{1-\cos(\vartheta)}.
\end{align}
An easy computation shows that
\[
\dfrac{d}{d\vartheta} \left( -\frac{\cos(\vartheta/2)}{\sin(\vartheta/2)} \right) = \frac{1}{1-\cos(\vartheta)}.
\]
We can then evaluate the last integral explicitly, leading to
\[
\bE[Y_n] = \frac{c_n}{2\pi}\, \frac{\cos(\arccos(u_n)/2)}{\sin(\arccos(u_n)/2)}.
\]
Recall that $\arccos(u_n) \sim c_n^{1/2}$ as $n \to \infty$. Consequently
\[
\sin(\arccos(u_n)/2) \sim \frac{c_n^{1/2}}{2} \quad \text{and} \quad \cos(\arccos(u_n)/2) \sim 1.
\]
Substituting we get
\[
\bE[Y_n] \simeq \frac{c_n}{2\pi}\, \frac{2}{c_n^{1/2}} = \frac{c_n^{1/2}}{\pi},
\]
and therefore
\[
\sum_{n \in \bN} \bE[Y_n] \simeq \sum_{n \in \bN} \frac{c_n^{1/2}}{\pi} < \infty,
\]
which is condition (ii). It remains to control $\Var[Y_n]$. Note that $Y_n \leq 1$, which implies $Y_n^2 \leq Y_n$. Hence
\[
\Var[Y_n] \leq \bE[Y_n]\big(1 - \bE[Y_n]\big) \simeq \bE[Y_n],
\]
and summing yields
\[
\sum_n \Var[Y_n] \leq \sum_n \bE[Y_n] < +\infty,
\]
which is condition (iii). All three conditions of Theorem \ref{kolmogorov 3-series} being satisfied, we conclude that
\[
\sum_{n \in \bN} X_n = \sum_{n \in \bN} \frac{c_n}{|1-\zeta_n|^{2}} < \infty \quad \text{almost surely.}
\]

\end{proof}

\begin{proof}[Proof of Corollary \ref{THM3}]
Suppose first that
\[
\sum_n c_n^{1/2}<\infty.
\]
Given a single realization $\omega$, let us call 
\[
E(\omega)=\{\zeta \in \bT \colon \varphi_\omega \text{ does not have a finite angular derivative at }\zeta\}.  
\]
By Theorem \ref{THM2}, for every fixed \(\xi\in\mathbb T\),
\[
\mathbb P(\xi\in E(\omega))=0.
\]
Hence, by Fubini,
\[
\mathbb E[m(E(\omega))]
=
\int_{\mathbb T}\mathbb P(\xi\in E(\omega))\,dm(\xi)
=0.
\]
Since \(m(E(\omega))\ge0\), it follows that \(m(E(\omega))=0\) almost surely. 
By the Aleksandrov's disintegration theorem \cite{Ale87},

\[
\int_{\mathbb T}\sigma_{\varphi,\alpha}(E(\omega))\,dm(\alpha) = m(E(\omega)) =0.
\]
Therefore,
\(
\sigma_{\varphi,\alpha}(E(\omega))=0
\)
for \(m\)-almost every \(\alpha\in\mathbb T\). For such \(\alpha\), the measure \(\sigma_{\varphi,\alpha}\) is carried by the set of
points where \(\varphi\) has a finite angular derivative. By the Clark atom
criterion \cite[Theorem 3.1]{Saksman}, the atoms of \(\sigma_{\varphi,\alpha}\) are precisely the points
\(
\xi\in\mathbb T
\)
such that
\[
\varphi(\xi)=\alpha
\qquad\text{and}\qquad
|\varphi'(\xi)|<\infty.
\]
Consequently, for \(m\)-almost every \(\alpha\), the Clark measure
\(\sigma_{\varphi,\alpha}\) is purely atomic, i.e. discrete.
Thus, almost surely, \(\sigma_{\varphi,\alpha}\) is discrete for
\(m\)-almost every \(\alpha\in\mathbb T\).

\medskip
For the second implication, we reason analogously.
\end{proof}

This result has implications on the spectrum of the unitary Clark operator $U_\alpha$.

\begin{lem}[\cite{Ross2013LENSLO}, Theorem 10.6]
The eigenvalues of the Clark operator $U_\alpha$ are the points $\zeta \in \mathbb{T}$ such that $\varphi(\zeta) = \alpha$ and $|\varphi'(\zeta)| < \infty$. In particular, the eigenvalues of $U_\alpha$ correspond exactly to the atoms of the Clark measure $\sigma_{\varphi, \alpha}$.
\end{lem}

The measure $\sigma_\alpha$ is the spectral measure of the Clark operator $U_\alpha$. The problem of understanding the properties of the Clark measure $\sigma_\beta$ given the $\sigma_\alpha$, is equivalent to the analysis of the interaction between different unitary Clark extensions $U_\alpha$ and $U_\beta$ for $\alpha \neq \beta$.  
According to Corollary \ref{THM3}, if $\sum_n c_n^{1/2}<\infty$, almost every unitary extension of the random compressed shift $S_{\varphi}$ have discrete spectral measure almost surely. On the other hand, if $\sum_n c_n^{1/2}=\infty$, then almost every unitary extension have singular continuous measure almost surely. Similar results for deterministic Clark measures have been obtained by Donoghue \cite{Don65}, and Aleksandrov \cite{Ale87}.

\section{Proof of Theorem \ref{THM4}}\label{sect:thm4}

In order to prove Theorem \ref{THM4} we use the following result, proved in \cite{Huangsaksman}*{Lemma 16}

\begin{lem}\label{lem:saksman}
    Let $\Lambda=(\lambda_k)_{k \in \bN}$ be the zeros of the Blaschke product $\varphi: \bD \to \bD$ and let $0 < \beta < 1$. Then
    \[
    \sum_{k \in \bN} (1-|\lambda_k|)^\beta < \infty \quad \iff \quad \int_{\bD} (1-|z|^2)^{\beta-2} \log \frac{1}{|\varphi(z)|} dA(z) < \infty.
    \]
\end{lem}

We recall the following well-known estimate, see \cite[Theorem 1.7]{KHZ} for a proof.

\begin{lem}\label{lem:Poisson_kernel}
    Let $\alpha \in (0,1)$. Then,
    \[
    \frac{1}{2\pi}\int_{-\pi}^\pi \left(\frac{1-r^2}{1-2r\cos(\vartheta)+r^2}\right)^\alpha d\vartheta \simeq
    \begin{dcases}
        (1-r)^{1-\alpha} & \text{if} \quad 1/2 < \alpha < 1\\
        (1-r)^{1/2}\log\frac{1}{1-r} & \text{if} \quad \alpha = 1/2 \\
        (1-r)^{\alpha} & \text{if} \quad 0 < \alpha < 1/2.
    \end{dcases}
    \]
\end{lem}

\begin{lem} \label{lem: expectation}
    Let $\varphi$ be the random inner function defined as in \eqref{phi:def}. Suppose that there exists $\alpha \in [1/2,1)$ such that $(c_n)_{n \in \bN} \in \ell^\alpha$. Then
    \[
    \sup_{|z|=r}\bE \left[ \log \frac{1}{|\varphi(z)|} \right] \lesssim (1-r)^{1-\alpha-\varepsilon}.
    \]
\end{lem}

\begin{proof}
   By the rotational invariance of $\mu$, we may restrict our attention to the case $z = r \in (0,1)$. A direct computation then gives
   \[
   \bE \Big[ \log \frac{1}{|\varphi(z)|} \Big] \leq \bE \Big[ \log \big(1+ \frac{4\Re H\mu(z)}{(\Re H\mu(z)-1)^2}\big)\Big],
   \]
   where $\Re H\mu$ denotes the Poisson extension of the random measure $\mu$. Passing to the distribution function and changing variables, we obtain
   \begin{align*}
       \bE \Big[\log \big( 1+ \frac{\Re H\mu(r)}{(\Re H\mu(r)-1)^2} \big)\Big] & = \int_0^\infty \bP \Big(  \log \big( 1+ \frac{\Re H\mu(r)}{(\Re H\mu(r)-1)^2} \big) > t \Big) dt \\
       & = \int_0^\infty \bP\Big( \frac{\Re H\mu(r)}{(\Re H\mu(r)-1)^2} > u \Big) \frac{du}{u+1} \\
       & = \int_0^\infty \bP( \Re H\mu(r) \in I(u)) \frac{du}{u+1},
   \end{align*}
   where $I(u)$ is the interval
   \[
   I(u)= \Big( 1 + \frac{1}{2u}- \sqrt{\frac{1}{u} + \frac{1}{4u^2}}, \ 1 + \frac{1}{2u} + \sqrt{\frac{1}{u} + \frac{1}{4u^2}} \Big).
   \]
   Setting $A_r= (1-r)^{-(6-4\alpha)}$, we split the integral at this threshold:
   \begin{align*}
       \int_0^\infty \bP( \Re H\mu(r) \in I(u)) \frac{du}{u+1} & = \int_0^{A_r} \bP( \Re H\mu(r) \in I(u)) \frac{du}{u+1} + \int_{A_r}^\infty \bP( \Re H\mu(r) \in I(u)) \frac{du}{u+1}\\
       & = A+B,
   \end{align*}
   and we estimate the two integrals separately.

   We begin with $A$. Writing $l(u) = 1 + \frac{1}{2u}- \sqrt{\frac{1}{u} + \frac{1}{4u^2}}$ for the left endpoint of $I(u)$, Markov's inequality together with Lemma \ref{lem:Poisson_kernel} yields
   \begin{align*}
       \bP(\Re H\mu(r) > l(u)) & \leq \bP(\Re H\mu(r)^\alpha > l(u)^\alpha) \\
       & \leq l(u)^{-\alpha}\bE[\Re H\mu(r)^\alpha]\\
       & \leq l(u)^{-\alpha} \sum_{n \in \bN} c_n^\alpha \bE \Big[ \big( \frac{1-r^2}{1+r^2-2r\cos(\vartheta_n)} \big)^\alpha \Big]\\
       & \lesssim l(u)^{-\alpha} \psi(r),
   \end{align*}
   where
   \[
   \psi(r) = 
   \begin{dcases}
       (1-r)^{1-\alpha} & \text{if } 1/2 < \alpha < 1,\\
       (1-r)^{1/2}\log\frac{1}{1-r} & \text{if } \alpha = 1/2.
   \end{dcases}
   \]
   Since $l(u) \simeq \min(u,1)$, splitting the range of integration leads to
   \begin{align}\label{eq:A}
       A & = \Big(\int_0^{1-r} + \int_{1-r}^{1} + \int_1^{A_r}\Big) \bP(\Re H\mu(r) \in I(u)) \frac{du}{u+1} \nonumber\\
       & \lesssim (1-r) + \int_{1-r}^1 \frac{\psi(r)}{l(u)^\alpha} \frac{du}{u+1} + \int_{1}^{A_r} \frac{\psi(r)}{l(u)^\alpha} \frac{du}{u+1} \nonumber \\
       & \simeq (1-r) + \psi(r) + \psi(r) \ln (A_r) \nonumber\\
       & \lesssim (1-r)^{1-\alpha - \varepsilon} \quad \text{for every } \varepsilon >0.
   \end{align}
   For $B$, Esséen's inequality gives
   \begin{align*}
       \bP(\Re H\mu(r) \in I(u)) & = \bP\Big(\big|\Re H\mu(r)-1-\frac{1}{2u} \big| \leq \sqrt{\frac{1}{u}+\frac{1}{4u^2}}\Big) \\
       & \leq C \sqrt{\frac{1}{\sqrt{u}(1-r)}},
   \end{align*}
   and hence
   \begin{align}\label{eq:B}
	   B & \lesssim \int_{A_r}^\infty \frac{1}{\sqrt{(1-r)}u^{\frac{1}{4}}} \frac{du}{u+1} \nonumber \\
       & \simeq (1-r)^{-1/2} \int_{A_r}^\infty \frac{du}{u^{5/4}} \nonumber \\
       & = (1-r)^{1-\alpha}.
   \end{align}
   Combining \eqref{eq:A} and \eqref{eq:B} yields the result.
\end{proof}

Note that the above estimate does not add useful information for $\alpha < 1/2$.

\bigskip

We can now prove Theorem \ref{THM4}.

\begin{proof}[Proof of Theorem \ref{THM4}]
    Let $\varepsilon>0$ and $0<\varepsilon_1< \varepsilon$. Then from Lemma \ref{lem: expectation} we have that for every $z \in \bD$
    \[
    \bE \left[ \log \frac{1}{|\varphi(z)|} \right] \lesssim (1-|z|)^{1-\alpha-\varepsilon_1}.
    \]
    We get
    \begin{align*}
        \bE \Big[ \int_{\bD} (1-|z|^2)^{\alpha+\varepsilon-2} \log \frac{1}{|\varphi(z)|}dA(z) \Big]
        & \lesssim \int_{\bD} (1-|z|^2)^{\alpha+\varepsilon-2} (1-|z|)^{1-\alpha-\varepsilon} dA(z) \\
        & \simeq \int_{\bD} (1-|z|)^{-1+(\varepsilon-\varepsilon_1)} dA(z) < \infty.
    \end{align*}
    This implies that for every $\varepsilon>0$
    \[
    \int_{\bD} (1-|z|^2)^{\alpha+\varepsilon-2} \log \frac{1}{|\varphi(z)|}dA(z) < \infty \quad \text{a.s.}
    \]
    Using Lemma \ref{lem:saksman} we obtain the result.
\end{proof}

To prove Theorem \ref{Prop1}, we again use Lemma \ref{lem:saksman}. Establishing our final result requires the following lemma. We wish to emphasize that this result appears to be ``folklore'' in the context of multifractal analysis and random coverings. Nevertheless, we were unable to find a result covering our specific situation, and for this reason we provide a proof of the following result. We refer the reader to some related works on the subject \cite{BarralSeuret08, Falconer00}.

\begin{lem}\label{G_set}
Let $\mu$ be defined as in \eqref{mu:def}. Suppose that \(c_n\) is a decreasing sequence such that
\begin{equation}\tag{LR}
\lim_{n\to\infty}\frac{-\log c_n}{\log n}=p,
\end{equation}
with $p\geq 1$.
Then, almost surely,
\[
G=\left\{\vartheta\in[0,2\pi]:
\lim_{\delta\to0}
\frac{\log \mu(e^{i\vartheta}-\delta,e^{i\vartheta}+\delta)}
{\log\delta}
=p
\right\}
\]
has full Lebesgue measure.
\end{lem}

\begin{proof}
It is enough to work on $[0,1)$ and consider
\[
\mu=\sum_{n\ge1}c_n\delta_{X_n},
\]
where the $X_n$ are i.i.d. uniformly distributed on $[0,1)$. For $\delta>0$, set
\[
M_\delta(x)=\mu((x-\delta,x+\delta))
=\sum_{n\ge1}c_n\mathbb 1_{\{|X_n-x|<\delta\}}.
\]
We prove the claim first for dyadic radii $\delta_j=2^{-j}$.

Fix $0<\eta < 1$. We show that, almost surely, for $m-$a.e. $x$, for all large $j$,
\[
\delta_j^{p+\eta}\le M_{\delta_j}(x)\le \delta_j^{p-\eta}.
\]

We start with the lower bound. By \eqref{log-regularity}, we can choose $0<\gamma<\eta$
such that for all large $n$,
\[
c_n\ge n^{-(p+\gamma)}.
\]

Choose $\tau>0$ so small that $(1+\tau)(p+\gamma)<p+\eta$,
and set $N_j=\lfloor \delta_j^{-1-\tau}\rfloor$.
Hence, for all large $j\geq j_0$,
\[
c_{N_j}\ge N_j^{-(p+\gamma)}
\ge \delta_j^{(1+\tau)(p+\gamma)}
\ge \delta_j^{p+\eta}.
\]
Let
\[
E_j=\{x\in(0,1): |x-X_n|\ge \delta_j
\text{ for every }1\le n\le N_j\}.
\]
For $x\in(\delta_j,1-\delta_j)$,
\[
\mathbb P(x\in E_j)=(1-2\delta_j)^{N_j}
\le e^{-2\delta_jN_j}.
\]
Thus, by Fubini's theorem,
\[
\mathbb E[m(E_j)]
\le 2\delta_j+e^{-2\delta_jN_j}.
\]
Since $\delta_jN_j\simeq \delta_j^{-\tau}$, the right-hand side is summable in
$j$. Therefore
\[
\sum_j m(E_j)<\infty \quad \text{almost surely.}
\]
By Borel--Cantelli, for a.e. $x$, one has $x\notin E_j$ for all
large $j\geq j_1$. Hence, for such $x$, there exists $n_j\le N_j$ with
$|X_{n_j}-x|<\delta_j$, and therefore
\[
M_{\delta_j}(x)\ge c_{n_j} \ge c_{N_j}\ge N_j^{-(p+\gamma)}  \geq \delta_j^{p+\eta}
\]
for all large $j$.

\medskip
We now prove the upper bound. Assume first that $p>1$. Choose $0<\gamma<\eta$ such that $p-\gamma>1$ and $\beta>0$ such that
\[
\theta=\eta-\gamma-\beta(p-1-\gamma)>0.
\]
Let
\(
K_j=\lfloor \delta_j^{-1+\beta}\rfloor
\)
 be the integer part.
The set of points lying within distance $\delta_j$ of one of
$X_1,\dots,X_{K_j}$ has Lebesgue measure at most
\[
2\delta_jK_j\le C\delta_j^\beta.
\]
Since $\sum_j\delta_j^\beta<\infty$, Borel--Cantelli implies that, for a.e.
$x$, the first $K_j$ atoms do not contribute to $M_{\delta_j}(x)$ for all large
$j\geq j_0$.

It remains to estimate the tail
\[
T_j(x)=\sum_{n>K_j}c_n\mathbb 1_{\{|X_n-x|<\delta_j\}}.
\]
Clearly
\[
m\{x:T_j(x)>\delta_j^{p-\eta}\}
\le
\delta_j^{-(p-\eta)}
\int_0^1T_j(x)\,dx.
\]
Since
\[
\int_0^1\mathbb 1_{\{|X_n-x|<\delta_j\}}\,dx\le 2\delta_j,
\]
we have
\[
\int_0^1T_j(x)\,dx 
\le 2\delta_j\sum_{n>K_j}c_n.
\]
Using again \eqref{log-regularity}, $c_n\le n^{-(p-\gamma)}$ for all large $n$, we obtain
\[
\sum_{n>K_j}c_n\le C K_j^{1-p+\gamma}.
\]
Thus
\[
m\{x:T_j(x)>\delta_j^{p-\eta}\}
\le
C\delta_j^{1-p+\eta}K_j^{1-p+\gamma}
\le C\delta_j^\theta.
\]
Since $\theta>0$, the last expression is summable in $j$. Another
Borel--Cantelli argument gives
\[
T_j(x)\le \delta_j^{p-\eta}
\]
for a.e. $x$ and all large $j\geq j_1$. Since the first $K_j$ atoms eventually do not
contribute, we obtain
\[
M_{\delta_j}(x)\le \delta_j^{p-\eta}
\]
for a.e. $x$ and all large $j\geq \max (j_0,j_1)$.

\medskip
If $p=1$, the upper bound follows directly from $\sum_n c_n<\infty$. Indeed,
\[
\int_0^1M_{\delta_j}(x)\,dx 
\le 2\delta_j\sum_{n\ge1}c_n,
\]
and 
\[
m\{x:M_{\delta_j}(x)>\delta_j^{1-\eta}\}
\le C\delta_j^\eta.
\]
This is summable in $j$, so Borel--Cantelli yields
\[
M_{\delta_j}(x)\le \delta_j^{1-\eta}
\]
for a.e. $x$ and all large $j$.

\medskip
Combining the lower and upper bounds, we have shown that, for every
$\eta>0$, almost surely, for a.e. $x$,
\[
\delta_j^{p+\eta}\le M_{\delta_j}(x)\le \delta_j^{p-\eta}
\]
for all large $j$. 
Taking logarithms, dividing by $\log\delta_j<0$, letting $\eta \to 0$ through a countable sequence, we obtain
\[
\lim_{j\to\infty}
\frac{\log M_{\delta_j}(x)}{\log\delta_j}
=p
\]
for a.e. $x$, almost surely.
Finally, the result for general, non dyadic $\delta$ follows from the monotonicity of
$\delta\mapsto M_\delta(x)$.
\end{proof}

Using the above lemma, we now prove Theorem \ref{Prop1}.

\begin{proof}[Proof of Theorem \ref{Prop1}]
We proceed by distinguishing between two cases.

\medskip

\underline{\textit{Case $p>1$}}.
Fix \(0<\delta<p-1\). For \(\varepsilon_0>0\), define
\[
E_{p,\delta}(\varepsilon_0)
=
\left\{
\xi\in\mathbb T:
\varepsilon^{p+\delta}
\leq \mu(I(\xi,\varepsilon))
\leq \varepsilon^{p-\delta}
\text{ for every }0<\varepsilon\leq \varepsilon_0
\right\},
\]
where
\[
I(\xi,\varepsilon)=\{\zeta\in\mathbb T:|\zeta-\xi|<\varepsilon\}.
\]
By Lemma \ref{G_set}, for \(\varepsilon_0>0\) small enough,
\[
m(E_{p,\delta}(\varepsilon_0))>0 \quad \text{almost surely}.
\]
Indeed, 
\[
G=
\left\{
\xi\in\mathbb T:
\lim_{\varepsilon\to0}
\frac{\log\mu(I(\xi,\varepsilon))}{\log\varepsilon}
=p
\right\},
\]
has full Lebesgue measure a.s. and
\[
G\subseteq \bigcup_{k=1}^{\infty}E_{p,\delta}(1/k).
\]
Hence, for some sufficiently small
\(\varepsilon_0\), the set \(E_{p,\delta}(\varepsilon_0)\) has positive
Lebesgue measure a.s.
Set
\[
E_0:=E_{p,\delta}(\varepsilon_0),
\]
and fix a realization $\omega \in \Omega$ such that $m(E_0)>0$. Define the truncated Korenblum region
\[
E_1:=
\left\{
z=re^{i\theta}\in\mathbb D:
1-r<\varepsilon_0,\ 
\operatorname{dist}(e^{i\theta},E_0)\leq 1-r
\right\}.
\]
We prove that, for \(z\in E_1\),
\[
\Re H\mu(z)\gtrsim (1-|z|)^{p-1+\delta}
\]
and
\[
|H\mu(z)|\leq C_{\varepsilon_0}.
\]
Let \(z=re^{i\theta}\in E_1\), and put
\[
h=1-r.
\]
By definition of \(E_1\), there exists \(\xi\in E_0\) such that
\[
|e^{i\theta}-\xi|\leq h.
\]

We first prove the lower bound. If \(\zeta\in I(\xi,h)\), then
\[
|\zeta-z|
\leq
|\zeta-\xi|+|\xi-e^{i\theta}|+|e^{i\theta}-re^{i\theta}|
\leq 3h.
\]
Moreover since $1-|z|^2 \geq h$, we have
\[
\frac{1-|z|^2}{|\zeta-z|^2}
\geq
\frac{1}{9h},
\qquad \zeta\in I(\xi,h).
\]
Since \(\xi\in E_0\) and \(h<\varepsilon_0\), we obtain
\[
\begin{aligned}
\Re H\mu(z)
&=
\int_{\mathbb T}
\frac{1-|z|^2}{|\zeta-z|^2}\,d\mu(\zeta) \\
&\geq
\int_{I(\xi,h)}
\frac{1-|z|^2}{|\zeta-z|^2}\,d\mu(\zeta) \geq
\frac{\mu(I(\xi,h))}{9h}
\geq
\frac{1}{9}h^{p-1+\delta}.
\end{aligned}
\]
Thus
\[
\Re H\mu(z)\gtrsim (1-|z|)^{p-1+\delta}.
\]

We now prove the upper bound for \(H\mu\). Since
\[
|z-\xi|
\leq
|z-e^{i\theta}|+|e^{i\theta}-\xi|
\leq 2h,
\]
we have \(z\in\Gamma_2(\xi)\), i.e. the Stoltz region with vertex at $\xi$ and opening $2$. Hence for every \(\zeta\in\mathbb T\),
\[
|\zeta-z|\gtrsim h+|\zeta-\xi|.
\]
Therefore
\[
|H\mu(z)|
\leq
\int_{\mathbb T}\frac{|\zeta+z|}{|\zeta-z|}\,d\mu(\zeta)
\lesssim
\int_{\mathbb T}\frac{d\mu(\zeta)}{h+|\zeta-\xi|}.
\]
Choose \(N\geq0\) such that
\[
2^Nh\leq \varepsilon_0<2^{N+1}h.
\]
We decompose \(\mathbb T\) into dyadic annuli centered at \(\xi\). First,
\[
\int_{I(\xi,h)}
\frac{d\mu(\zeta)}{h+|\zeta-\xi|}
\leq
\frac{\mu(I(\xi,h))}{h}
\leq
h^{p-1-\delta}.
\]
For \(1\leq k\leq N\), put
\[
A_k=I(\xi,2^kh)\setminus I(\xi,2^{k-1}h).
\]
Then
\[
\begin{aligned}
\int_{A_k}
\frac{d\mu(\zeta)}{h+|\zeta-\xi|}
&\leq
\frac{\mu(I(\xi,2^kh))}{2^{k-1}h} \\
&\lesssim
\frac{(2^kh)^{p-\delta}}{2^kh} =
2^{k(p-1-\delta)}h^{p-1-\delta}.
\end{aligned}
\]
Since \(p-1-\delta>0\), we get
\[
\begin{aligned}
\sum_{k=0}^{N}
\int_{A_k}
\frac{d\mu(\zeta)}{h+|\zeta-\xi|}
&\lesssim h^{p-1-\delta}
\sum_{k=0}^{N}2^{k(p-1-\delta)} \\
&\lesssim
 h^{p-1-\delta}2^{N(p-1-\delta)} \\
&= (2^Nh)^{p-1-\delta}
\leq
\varepsilon_0^{p-1-\delta}.
\end{aligned}
\]
It remains to estimate the part outside \(I(\xi,2^Nh)\). Since
\[
2^Nh>\frac{\varepsilon_0}{2},
\]
we have
\[
h+|\zeta-\xi|\geq |\zeta-\xi|\geq \frac{\varepsilon_0}{2}
\]
on \(\mathbb T\setminus I(\xi,2^Nh)\). Hence
\[
\int_{\mathbb T\setminus I(\xi,2^Nh)}
\frac{d\mu(\zeta)}{h+|\zeta-\xi|}
\leq
\frac{2}{\varepsilon_0}.
\]
Combining the estimates gives
\[
|H\mu(z)|
\lesssim
\left(
\varepsilon_0^{p-1-\delta}
+
\varepsilon_0^{-1}
\right).
\]
The right-hand side is independent of \(z\in E_1\). Therefore \(H\mu\)
is bounded on \(E_1\).

To conclude the proof, we argue as in \cite{Huangsaksman}. On
\(E_1\) we have
\[
\Re H\mu(z)\gtrsim (1-|z|)^{p-1+\delta}
\]
and
\[
|H\mu(z)|\leq C_{\varepsilon_0}.
\]
Hence
\[
(\Re H\mu(z)-1)^2+(\Im H\mu(z))^2
\leq C
\]
on \(E_1\). Therefore for $z \in E_1$
\[
\log\frac1{|\varphi(z)|}
=
\frac12
\log\left(
1+\frac{4\Re H\mu(z)}
{(\Re H\mu(z)-1)^2+(\Im H\mu(z))^2}
\right)
\gtrsim \Re H\mu(z) \gtrsim (1-|z|)^{p-1+\delta}
\]
It follows that
\[
\begin{aligned}
\int_{\mathbb D}
(1-|z|)^{-p-\varepsilon}
\log\frac1{|\varphi(z)|}\,dA(z)
&\geq
\int_{E_1}
(1-|z|)^{-p-\varepsilon}
\log\frac1{|\varphi(z)|}\,dA(z) \\
&\gtrsim
\int_{E_1}
(1-|z|)^{-1-\varepsilon+\delta}\,dA(z).
\end{aligned}
\]
Since \(E_0\) has positive Lebesgue measure and
\[
\{re^{i\theta}:e^{i\theta}\in E_0,\ 1-r<\varepsilon_0\}
\subset E_1,
\]
we obtain
\[
\int_{E_1}
(1-|z|)^{-1-\varepsilon+\delta}\,dA(z)
\geq
m(E_0)
\int_{1-\varepsilon_0}^{1}
(1-r)^{-1-\varepsilon+\delta}\,r\,dr.
\]
The last integral diverges whenever \(\delta<\varepsilon\). Therefore
\[
\int_{\mathbb D}
(1-|z|)^{-p-\varepsilon}
\log\frac1{|\varphi(z)|}\,dA(z)
=+\infty.
\]
By Lemma \ref{lem:saksman}, this implies
\[
\sum_k (1-|\lambda_k|)^{2-p-\varepsilon}
=
+\infty.
\]
Since what we have showed holds for a.e. $\omega \in \Omega$, we proved the theorem.

\medskip

\underline{\textit{Case $p=1$}}. Fix \(\varepsilon>0\) and choose
\(
0<\eta<{\varepsilon}/{3}
\). 
We work as in the case $p> 1$. We consider $z \in E_1$.
As before, the lower estimate for the Poisson integral gives
\[
\Re H\mu(z)
=
\int_{\mathbb T}
\frac{1-|z|^2}{|\zeta-z|^2}\,d\mu(\zeta)
\gtrsim h^\eta.
\]
The estimate of \(H\mu\) from above is slightly different, but reasoning as in the case $p>1$, we get that 
\[
|H\mu(z)|\lesssim h^{-\eta},
\qquad z\in E_1.
\]
Since on \(E_1\),
\[
\Re H\mu(z)\gtrsim h^\eta,
\]
then
\[
\log\frac1{|\varphi(z)|}
=
\frac12
\log\left(
1+
\frac{4 \Re H\mu(z)}
{(\Re H \mu(z)-1)^2+\Im H \mu(z)^2}
\right) \gtrsim h^{3\eta}.
\]
It follows that
\[
\begin{aligned}
\int_{\mathbb D}
(1-|z|)^{-1-\varepsilon}
\log\frac1{|\varphi(z)|}\,dA(z)
&\geq
\int_{E_1}
(1-|z|)^{-1-\varepsilon}
\log\frac1{|\varphi(z)|}\,dA(z) \\
&\gtrsim
\int_{E_1}
(1-|z|)^{-1-\varepsilon+3\eta}\,dA(z)\\
&\geq
m(E_0)
\int_{1-\varepsilon_0}^{1}
(1-r)^{-1-\varepsilon+3\eta}r\,dr.
\end{aligned}
\]
The last integral diverges because
\(
3\eta<\varepsilon
\).
By Lemma \ref{lem:saksman}, it follows that
\[
\sum_k (1-|\lambda_k|)^{1-\varepsilon}=\infty.
\]
Once again, since everything holds for a.e. $\omega \in \Omega$, we get the result.
    
\end{proof}

\section*{Methodology}
In a private communication, Professor S. Seuret informed us that Lemma \ref{G_set} is a folklore result. We have not been able to locate a precise reference for the claim and the proof given in the article has been obtained with the help of ChatGPT 5.5. Furthermore the same large language model has been used for proof-reading of the article and for bibliographical research. 
All the authors of the present paper adhere to the principles of the ``Leiden Declaration on Artificial Intelligence and Mathematics''.

\section*{Acknowledgements}
We would like to thank Professor S. Seuret for useful suggestions regarding the local H\"older exponent of random measures.

\bibliographystyle{habbrv}
\bibliography{biblio}

\end{document}